\documentclass[10pt, a4paper]{article}

\usepackage{amssymb, amsmath, amsthm, mathrsfs, color}
\usepackage{mathbbol, graphicx, latexsym, tkz-graph,mathtools}
\usepackage[nobysame]{amsrefs}

\usepackage[margin = 3 cm]{geometry}

\newtheorem{theorem}{Theorem}[section]
\newtheorem{lemma}[theorem]{Lemma}
\newtheorem{corollary}[theorem]{Corollary}

\theoremstyle{remark}

\newcommand{\cB}	{\mathcal B}
\newcommand{\bM}	{\mathbb M}

\begin{document}
\title{Steiner's Porism in finite Miquelian M\"obius planes}
\author{Norbert Hungerb\"uhler (ETH Z\"urich) \\ Katharina Kusejko (ETH Z\"urich)}
      
\date{}
\maketitle

\begin{abstract}\noindent 
  We investigate Steiner's Porism in finite Miquelian M\"obius planes
  constructed   over  the   pair  of   finite  fields   $GF(p^m)$  and
  $GF(p^{2m})$, for  $p$ an odd prime  and $m \geq 1$.   Properties of
  common  tangent  circles  for   two  given  concentric  circles  are
  discussed and  with that, a  finite version of Steiner's  Porism for
  concentric circles is stated and proved.  We formulate conditions on
  the  length of  a Steiner  chain by  using the quadratic residue
  theorem in  $GF(p^m)$.  These  results are  then generalized  to an
  arbitrary pair of non-intersecting circles by introducing the notion
  of  capacitance, which  turns  out to  be  invariant under  M\"obius
  transformations.  Finally,   the  results  are  compared   with  the
  situation in the classical Euclidean plane.
\end{abstract}



\section*{Introduction} 

In the  19th century,  the Swiss mathematician  Jakob Steiner  (1796 -
1863) discovered a beautiful  result about mutually tangential circles
in the  Euclidean plane,  also known  as \emph{Steiner's  Porism}. One
version reads as follows.

\begin{theorem}[Steiner's Porism] 
  Let $\cB_1$ and $\cB_2$ be disjoint circles in the Euclidean plane.
  Consider a sequence of different
  circles $\mathcal{T}_1,\ldots,\mathcal{T}_k$, 
  which are tangential to both $\cB_1$ and $\cB_2$.
  Moreover, let $\mathcal{T}_i$ and $\mathcal{T}_{i+1}$ be tangential 
  for all $i=1,\ldots,k-1$.
  If $\mathcal{T}_1$ and $\mathcal{T}_k$ are tangential as well, 
  then there are infinitely many such chains.
  In particular, every chain of consecutive tangent circles closes after $k$ steps.
\end{theorem}

\begin{figure}[h!] \label{fig: Steiner Porism}
  \centering
  \includegraphics[width=6cm]{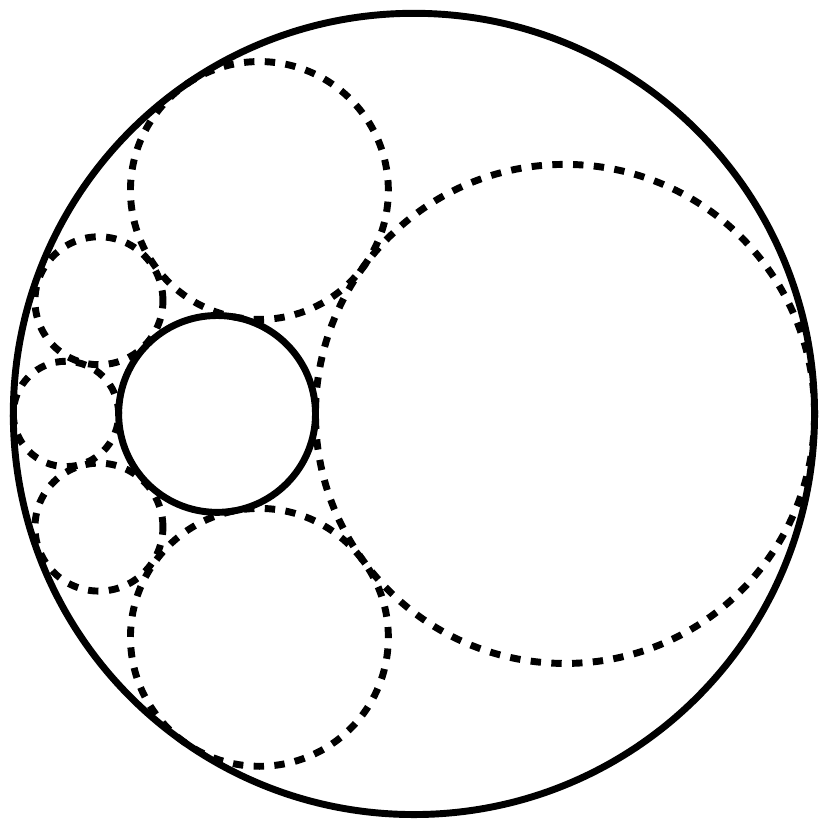} \qquad \includegraphics[width=6.4cm]{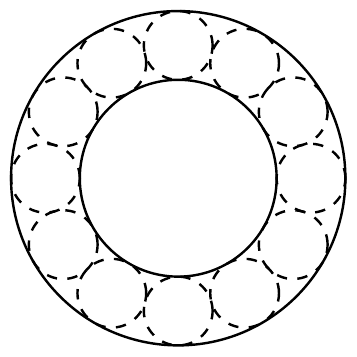}
  \caption{Two examples of Steiner chains in the Euclidean plane.}
\end{figure}

Steiner  thoroughly  investigated  such  chains and  found  many  nice
properties.   For example,  he  could prove  that  the mutual  tangent
points of  the circles  $\mathcal{T}_1,\ldots,\mathcal{T}_k$ lie  on a
circle and their centers  on a conic whose foci are  the centers of the
initial circles $\cB_1$ and $\cB_2$.  He studied conditions for such a
chain to close after $k$ steps in  terms of the radii and the distance
of the centers of $\cB_1$ and  $\cB_2$.  The interested reader can refer
to  \cite{MR0389515},  \cite{MR1017034}  or \cite{MR990644}  for  more
information on Steiner's original result.

In recent years, some refinements and generalizations of Steiner's Porism were studies.
For example in \cite{MR2877262}, Steiner chains with rational radii are discussed and
in \cite{MR3193739}, a three dimensional analogue of Steiner's Porism is presented.

Porisms  in finite  geometry have  not been  investigated to  the same
extent as  in the Euclidean  case. In particular,  as far as  we know,
Steiner's Porism  was not  yet considered  in finite  M\"obius planes.
However,  closing   chains  of  touching  circles   with  a  different
arrangement have been investigated in~\cite{MR758854}.

M\"obius planes consist of points $\mathbb{P}$  and circles $\mathbb{B}$,
which satisfy three axioms.  First, there  needs to be a unique circle
through three  given points.   Second, there  exists a  unique tangent
circle through  a point  on a  given circle  and a  point not  on this
circle.   Finally, a  richness axiom  ensures  that the  plane is  not
trivial.  More precisely, the three axioms read as follows.
\begin{enumerate}
  \item[(M1)] 
    For  any three elements  $P,Q, R \in  \mathbb{P}$, $P \neq  Q$, $P \neq  R$ and $Q \neq  R$, there
    exists  a  unique element  $g   \in  \mathbb{B}$  with  $P \in g$, $Q \in g$ and $R \in g$.
  \item[(M2)] 
    For  any $g \in  \mathbb{B}$, $P,Q   \in  \mathbb{P}$ with  $P \in g$ and $Q \notin g$, 
    there exists a unique element $h \in  \mathbb{B}$ such that $P \in h$ and $Q \in h$, 
    but for all $R \in \mathbb{P}$ with $R \in g, P \neq R$, we have $R \notin h$.
  \item[(M3)] 
    There are four  elements $P_1, P_2, P_3, P_4 \in \mathbb{P}$ such that for all $g \in \mathbb{B}$, 
    we have $P_i \notin g$ for at least one $i \in \{ 1,2,3,4 \}$.
\end{enumerate}
In the  present paper, we will  look at Steiner's Porism  in Miquelian
M\"obius  planes.   They  are  the classical  finite  models  for  the
M\"obius axioms and are constructed over the finite field $GF(p^m)$ of
order $p^m$, for $p$ an odd prime and $m \geq 1$.  The resulting plane
is  denoted   by  $\bM(p^m)$,  the   details  are  explained   in  the
preliminaries.   We will  start with  two disjoint  concentric circles
$\cB_a$  and  $\cB_b$  with  radii   $a$  and  $b$,  respectively,  in
$\bM(p^m)$.  We look for conditions  and properties of their potential
common  tangent  circles.   Concerning  this  question,  we  find  the
following:

\begin{theorem}[cf. Theorem \ref{th: exactly two tangent circles}]
  Let $\frac{b}{a}\neq 1$ be a square in $GF(p^m)$. 
  For every point $P$ on $\cB_a$ there are exactly two circles $g$ and $h$, 
  which are tangential to $\cB_a$ in $P$ and tangential to $\cB_b$ in $\mu P$ and $-\mu P$, 
  respectively, where $\mu^2 = \frac{b}{a}$.
\end{theorem}

We are interested in finding a condition for the existence of Steiner chains, i.e.\
a chain of circles which are tangential to both $\cB_a$ and $\cB_b$ and also mutually tangential.
We will see that the finiteness directly implies that all chains close up
and we only have to deal with possibly degenerate Steiner chains, which will be defined later on.
In particular, our finite version of Steiner's Porism reads as follows:
 
\begin{theorem}[cf. Theorem \ref{th: finite Steiner}]
  Consider the two disjoint concentric circles $\cB_1$ and $\cB_b$,
  with radii $1$ and $b$, respectively, in $\bM(p^m)$. 
  Assume we can find two circles $g$ and $h$ tangential to $\cB_b$, 
  which are tangential to $\cB_1$ in $1$ and $P$, respectively,
  and $g$ and $h$ are mutually tangential.
  Then a Steiner chain of length $k$ can be constructed, 
  where $k \in \{1, \ldots, p^m+1 \}$ is minimal with $ P^k = 1$.
  Moreover, such a Steiner chain of length $k$ can be constructed starting with any
  point $Q$ on $\cB_1$.
\end{theorem}

In the classical Euclidean plane, 
conditions on the length of Steiner chains are well-known.
For two concentric circles of radius $1$ and $R$, 
one can construct a Steiner chain of length $k \geq 3$
which wraps $w$ times around the smaller circle, if and only if
$$ R = \frac{1 + \sin(\phi)}{1-\sin(\phi)},$$
where $\phi = \frac{w\pi}{k}$.
We will ask for such conditions in the finite case and obtain the following result:

\begin{theorem}[cf. Theorem \ref{concentric}]
  Let $b = \mu^2\neq 1$ for $\mu$ in $GF(p^m)$.
  The two concentric circles $\cB_1$ and $\cB_b$ with radius $1$ and $b$, respectively,
  carry a proper Steiner chain of length $k$ if and only if the following conditions are satisfied.
    \begin{enumerate}
      \item $-\mu$ is a non-square in $GF(p^m)$,
      \item $\mu$ solves $ P^k=1$ for $P$ given by 
        \begin{align}
          P = \frac{-\mu^2+6\mu-1 + 4(\mu-1)\sqrt{-\mu}}{(1+\mu)^2},
        \end{align}
            but $P^l \neq 1$ for all $1 \leq l \leq k-1$.
    \end{enumerate}

In  particular,  if these  conditions  are  satisfied, $\cB_1$  and
$\cB_b$ carry a chain of length $k$.
\end{theorem}

In the last section, we introduce the notion of capacitance for a pair of
circles  and prove  that  this quantity  is  invariant under  M\"obius
transformations. This  fact allows to  formulate a criterion  for the
existence of proper Steiner chains of  length $k$ for an arbitrary pair
of non-intersecting circles.

Finally, the results are compared to the conditions on Steiner chains in the Euclidean plane.


\section{Preliminaries}

We will describe an explicit construction of the finite Miquelian M\"obius plane 
using finite fields.
For that, we need to recall some properties of finite fields $GF(p^m)$, 
$p$ an odd prime and $m \geq 1$. 

An element $a \in GF(p^m)$ is called a \emph{square} in $GF(p^m)$ 
if there exists some $b \in GF(p^m)$ with $a = b^2$. 
Otherwise, $a \in GF(p^m)$ is called a \emph{non-square} in $GF(p^m)$. 
Exactly half of the elements in $GF(p^m)\setminus \{0\}$ are squares. 
Note that the squares of $GF(p^m)$ form a subgroup of $GF(p^m)$, but the non-squares do not. 
In particular, multiplying two non-squares in $GF(p^m)$ gives a square 
and multiplying a square and a non-square in $GF(p^m)$ gives a non-square.

For any non-square $x$ in $GF(p^m)$, 
we can construct an extension field of $GF(p^m)$ 
by adjoining some element $\alpha$ with $\alpha^2 = x$ to $GF(p^m)$. 
The elements in the extension field $GF(p^m)(\alpha)$ are of the form $a+\alpha b$ for $a,b \in GF(p^m)$. 
Note that all elements of $GF(p^m)$ are squares in $GF(p^m)(\alpha)$. 
To see this, take some element $a \in GF(p^m)$. 
If $a$ is a square in $GF(p^m)$, it clearly is a square in $GF(p^m)(\alpha)$ as well. 
If $a$ is a non-square in $GF(p^m)$,
then $a x$ is a square in $GF(p^m)$ and hence $ax=b^2$ which leads to $a = \alpha^{-2} b^2$, 
i.e.\ $a$ is a square in $GF(p^m)(\alpha)$.

Since $GF(p^m)(\alpha)$ is isomorphic to any field with $p^{2m}$ elements, 
we will denote it by $GF(p^{2m})$. 

For $z \in GF(p^{2m})$, define the \emph{conjugate element} of $z$ over $GF(p^m)$ by
  $$\overline{z}:=z^{p^m}.$$
Note that $z=\overline{z}$ if and only if $z \in GF(p^m)$.
Define the \emph{trace} of $z$ over $GF(p^m)$ by
  $$ \operatorname{Tr}_{GF(p^{2m})/GF(p^m)}(z):=z+\overline{z} $$
and the \emph{norm} of $z$ over $GF(p^m)$ by 
  $$ \operatorname{N}_{GF(p^{2m})/GF(p^m)}(z):=z\overline{z}, $$
where we omit the subscript $GF(p^{2m})/GF(p^m)$ for notational convenience.
Recall that $\operatorname{Tr}(z)$ and $\operatorname{N}(z)$ are always in $GF(p^m)$ and that
$\overline{z_1+z_2}=\overline{z_1}+\overline{z_2}$ 
as well as $\overline{z_1 z_2}=\overline{z_1} \ \overline{z_2}$.
For more background on finite fields, 
one can have a look at \cite{MR1429394}.

We are now going to describe finite Miquelian M\"obius planes,
constructed over the pair of finite fields $GF(p^{m})$ and $GF(p^{2m})$, 
$p$ an odd prime and $m \geq 1$.
Such planes will be denoted by $\bM(p^m)$ and 
$p^m$ is called the \emph{order} of $\bM(p^m)$.

The $p^{2m}+1$ points of $\bM(p^m)$ are given by all elements of $GF(p^{2m})$
together with a point at infinity, denoted by $\infty$.
We distinguish between two different types of circles. 
For circles of the first type, we consider solutions of the equation
$N(z-s) = c$, i.e.\
\begin{align} \label{circle1}
\cB^1_{(s,c)}: \ (z-s)(\overline{z}-\overline{s}) = c
\end{align}
for $s \in GF(p^{2m})$ and $c \in GF(p^m)\backslash \{0\}$. 
It can easily be seen that there are $p^m+1$ points in $GF(p^{2m})$ 
on every circle (\ref{circle1}). 
Moreover, there are $p^{2m}(p^m-1)$ circles of the first type.

For circles of the second type, we consider the equation
$Tr(\overline{s}z) = c$, i.e.\
  \begin{align} \label{circle2}
    \cB^2_{(s,c)}: \ \overline{s}z+s\overline{z}=c
  \end{align}
for $s \in GF(p^{2m})\backslash \{0\}$ and $c \in GF(p^m)$.
For every such choice of $s$ and $c$, equation (\ref{circle2}) has $p^m$ solutions in $GF(p^{2m})$. 
To obtain circles of the second type, 
we take those $p^m$ solutions together with $\infty$. 
There are $(p^{2m}-1)p^m$ choices for $s$ and $c$, but scaling with any element of $GF(p^m)\backslash\{0\}$ leads to the same circle. 
Hence, there are $p^m(p^m+1)$ circles of the second type. 
There are $p^{3m}+p^m$ circles in total and 
on each circle there are $p^m+1$ points. 
This can also be seen by (M1), 
as three points uniquely define a circle.
Now, let $a,b,c,d \in GF(p^{2m})$ such that $ad-bc \neq 0$. 
The map $\Phi$ defined by
   $$ \Phi: \bM(p^{m}) \rightarrow \bM(p^{m}), \ \Phi(z)=\begin{cases}
  \frac{az+b}{cz+d}&\text{if $z\neq\infty$ and $cz+d\neq0$}\\
  \infty   &\text{if $z\neq\infty$ and $cz+d=0$}\\
  \frac ac &\text{if $z=\infty$ and $c\neq 0$}\\
  \infty   &\text{if $z=\infty$ and $c= 0$}\end{cases}$$
is called a \emph{M\"obius transformation} of $\bM(p^{m})$.
Every M\"obius transformation is an automorphism of $\bM(p^m)$.
A M\"obius transformation of the form $\Phi(z) = \frac{1}{z}$ 
is called an \emph{inversion} and M\"obius planes with an inversion are called
inversive M\"obius planes. In \cite{MR0177345} it is shown that
inversive M\"obius planes are exactly the Miquelian M\"obius planes.

Note that a M\"obius transformation operates three times sharply transitive, 
i.e.\ there is a unique M\"obius transformation mapping any three points into any other three given points.
For more background information on finite M\"obius planes, one can refer to \cite{MR1434062}.


\section{The plane $\bM(5)$} \label{section: Example}

We have a closer look at the M\"obius plane $\bM(5)$ 
constructed over $GF(5)$ and $GF(5^2)$ as described in the preliminaries. 
The first step is to start with two disjoint circles and search for common tangent circles. 
Since there is a M\"obius transformation 
which maps any given circle into $B_1:=\cB^1_{(0,1)}$, where
  $$ \cB_1 = \{1, 4, 2 + \alpha, 3 + \alpha, 2 + 4 \alpha, 3 + 4 \alpha\}, $$
we take this circle and search for circles disjoint to $\cB_1$. 
By direct computation, we find $30$ of them.

There are $10$ circles in the intersection of the set of all circles tangential to $\cB_1$ 
and all circles tangential to one of the $30$ circles disjoint to $\cB_1$. 
Moreover, each of these $10$ circles and $\cB_1$ have exactly $12$ common tangent circles. 
Consider for example
  $$ \cB_4:=\cB^1_{(0,4)} = \{2, 3, 1 + 2 \alpha, 4 + 2 \alpha, 1 + 3 \alpha, 4 + 3 \alpha\}. $$
Then the $12$ common tangent circles of $\cB_1$ and $\cB_4$ are given by
  \begin{align*}
    \mathcal{T}_1:=\cB^1_{(4,4)} &= \{1, 2, 2 \alpha, 3 + 2 \alpha, 3 \alpha, 3 + 3 \alpha\}\\
    \mathcal{T}_2:=\cB^1_{(2+\alpha,4)} &= \{\alpha, 4 + \alpha, 1 + 3 \alpha, 3 + 3 \alpha, 1 + 4 \alpha, 3 + 4 \alpha\}\\
    \mathcal{T}_3:=\cB^1_{(3+\alpha,4)} &= \{\alpha, 1 + \alpha, 2 + 3 \alpha, 4 + 3 \alpha, 2 + 4 \alpha, 4 + 4 \alpha\}\\
    \mathcal{T}_4:=\cB^1_{(1,4)} &= \{3, 4, 2 \alpha, 2 + 2 \alpha, 3 \alpha, 2 + 3 \alpha\}\\
    \mathcal{T}_5:=\cB^1_{(3+4\alpha,4)} &= \{2 + \alpha, 4 + \alpha, 2 + 2 \alpha, 4 + 2 \alpha, 4 \alpha, 1 + 4 \alpha\}\\
    \mathcal{T}_6:=\cB^1_{(2+4\alpha,4)} &= \{1 + \alpha, 3 + \alpha, 1 + 2 \alpha, 3 + 2 \alpha, 4 \alpha, 4 + 4 \alpha\}
  \end{align*} 
and
  \begin{align*}
    \mathcal{T}_7:=\cB^1_{(1+3\alpha,1)} &= \{3 + 2 \alpha, 4 + 2 \alpha, 3 \alpha, 2 + 3 \alpha, 3 + 4 \alpha, 4 + 4 \alpha\}\\
    \mathcal{T}_8:=\cB^1_{(4+3\alpha,1)} &= \{1 + 2 \alpha, 2 + 2 \alpha, 3 \alpha, 3 + 3 \alpha, 1 + 4 \alpha, 2 + 4 \alpha\}\\
    \mathcal{T}_9:=\cB^1_{(3,1)} &= \{2, 4, \alpha, 1 + \alpha, 4 \alpha, 1 + 4 \alpha\}\\
    \mathcal{T}_{10}:=\cB^1_{(4+2\alpha,1)} &= \{1 + \alpha, 2 + \alpha, 2 \alpha, 3 + 2 \alpha, 1 + 3 \alpha, 2 + 3 \alpha\}\\
    \mathcal{T}_{11}:=\cB^1_{(1+2\alpha,1)} &= \{3 + \alpha, 4 + \alpha, 2 \alpha, 2 + 2 \alpha, 3 + 3 \alpha, 4 + 3 \alpha\}\\
    \mathcal{T}_{12}:=\cB^1_{(2,1)} &= \{1, 3, \alpha, 4 + \alpha, 4\alpha, 4 + 4 \alpha\}.
  \end{align*}

Note that $\mathcal{T}_1$ is tangential to $\cB_1$ in $1$ and tangential to $\cB_4$ in $2$. 
Next, consider $\mathcal{T}_2$, 
which is tangential to $\cB_1$ in $3+4\alpha$ and tangential to $\cB_4$ in $1+3\alpha$. 
Note that $\mathcal{T}_1$ and $\mathcal{T}_2$ only intersect in $3+3\alpha$, 
i.e. they are mutually tangential. 
Having a closer look at $\mathcal{T}_2$, 
we see that only two of the $12$ circles above are tangential to $\mathcal{T}_2$ 
in points not on $\cB_1$ or $\cB_4$, 
namely $\mathcal{T}_1$, which we already considered, 
and $\mathcal{T}_3$, which is tangential to $\mathcal{T}_2$ in $\alpha$. 
Apparently, from now on, 
there is a unique way of constructing a chain of common tangent circles of $\cB_1$ and $\cB_4$ 
which are mutually tangential as well. 
Proceeding, we find that $\mathcal{T}_4$ is tangential to $\mathcal{T}_3$ in $2+3\alpha$. 
Then $\mathcal{T}_5$ is tangential to $\mathcal{T}_4$ in $2+2\alpha$. 
Finally, we find that $\mathcal{T}_6$ is tangential to $\mathcal{T}_5$ in $4\alpha$ and also
$\mathcal{T}_1$ is tangential to $\mathcal{T}_6$ in $3+2\alpha$, 
which closes the chain of circles. 

Note that those six tangent points lie on a circle itself, namely on
  $$ \cB_2:=\cB^1_{(0,2)} = \{\alpha, 2 + 2 \alpha, 3 + 2 \alpha, 2 + 3 \alpha, 3 + 3 \alpha, 4 \alpha\}.$$
Summarized, we denote this chain by
  $$
  \mathcal{T}_1 \xrightarrow{\parbox{12mm}{\centering\scriptsize$3+3\alpha$}}
   \mathcal{T}_2 \xrightarrow{\parbox{12mm}{\centering\scriptsize$\alpha$}}
   \mathcal{T}_3 \xrightarrow{\parbox{12mm}{\centering\scriptsize$2+3\alpha$}}
   \mathcal{T}_4 \xrightarrow{\parbox{12mm}{\centering\scriptsize$2+2\alpha$}}
      \mathcal{T}_5 \xrightarrow{\parbox{12mm}{\centering\scriptsize$4\alpha$}}
   \mathcal{T}_6 \xrightarrow{\parbox{12mm}{\centering\scriptsize$3+2\alpha$}} \mathcal{T}_1. 
   $$
Note that for the above chain, 
we only used six out of the twelve common tangent circles of $\cB_1$ and $\cB_4$, 
so let us start with a tangent circle not used so far, e.g.  $\mathcal{T}_7$. 
We find the chain
$$
\mathcal{T}_7 \xrightarrow{\parbox{12mm}{\centering\scriptsize$3\alpha$}}
\mathcal{T}_8 \xrightarrow{\parbox{12mm}{\centering\scriptsize$1+4\alpha$}} 
\mathcal{T}_9 \xrightarrow{\parbox{12mm}{\centering\scriptsize$1+\alpha$}}
\mathcal{T}_{10} \xrightarrow{\parbox{12mm}{\centering\scriptsize$2\alpha$}}
\mathcal{T}_{11} \xrightarrow{\parbox{12mm}{\centering\scriptsize$4+\alpha$}}
\mathcal{T}_{12} \xrightarrow{\parbox{12mm}{\centering\scriptsize$4+4\alpha$}}
\mathcal{T}_7. 
$$
Again, the six tangent points form a circle, namely
  $$ \cB_3:=\cB^1_{(0,3)} = \{1 + \alpha, 4 + \alpha, 2 \alpha, 3 \alpha, 1 + 4 \alpha, 4 + 4 \alpha\}.  $$

Now, let us proceed from here and look at the circles through two consecutive 
(where the order is defined by the chain before) 
points of $\cB_1$ and the according tangent point on $\cB_2$. 
We obtain six new circles, which are all tangential to $\cB_2$, given by:
  \begin{align*}
    \mathcal{S}_1:=\cB^1_{(1+\alpha,2)} &= \{1, 1 + 2 \alpha, 3 + 3 \alpha, 4 + 3 \alpha, 3 + 4 \alpha, 4 + 4 \alpha\}\\
    \mathcal{S}_2:=\cB^1_{(2\alpha,2)} &= \{2, 3, \alpha, 3 \alpha, 2 + 4 \alpha, 3 + 4 \alpha\}\\
    \mathcal{S}_3:=\cB^1_{(4+\alpha,2)} &= \{4, 4 + 2 \alpha, 1 + 3 \alpha, 2 + 3 \alpha, 1 + 4\alpha, 2 + 4 \alpha\}\\
    \mathcal{S}_4:=\cB^1_{(4+4\alpha,2)} &= \{4, 1 + \alpha, 2 + \alpha, 1 + 2 \alpha, 2 + 2 \alpha, 4 + 3 \alpha\}\\
    \mathcal{S}_5:=\cB^1_{(3\alpha,2)} &= \{2, 3, 2 + \alpha, 3 + \alpha, 2 \alpha, 4 \alpha\}\\
    \mathcal{S}_6:=\cB^1_{(1+4\alpha,2)} &= \{1, 3 + \alpha, 4 + \alpha, 3 + 2 \alpha, 4 + 2 \alpha, 1 + 3 \alpha\}
  \end{align*}

These six circles form a chain of mutually tangential circles as well. 
Moreover, there is a unique circle, except $\cB_2$, 
which is tangential to all of those six new circles, namely $\cB_3$.

We can do the same procedure once more, 
i.e.\ we consider circles through two consecutive points on $ \cB_3$ and the according tangent point of the chain in consideration. 
When doing that, note that we obtain six common tangent circles of $\cB_1$ and $\cB_4$, 
i.e.\ we have two chains including all twelve common tangent circles of of $\cB_1$ and $\cB_4$.

The following figure summarizes the above discussion.

\begin{figure}[h!]\label{fig: Example 5}
  \centering
  \includegraphics[width=14cm]{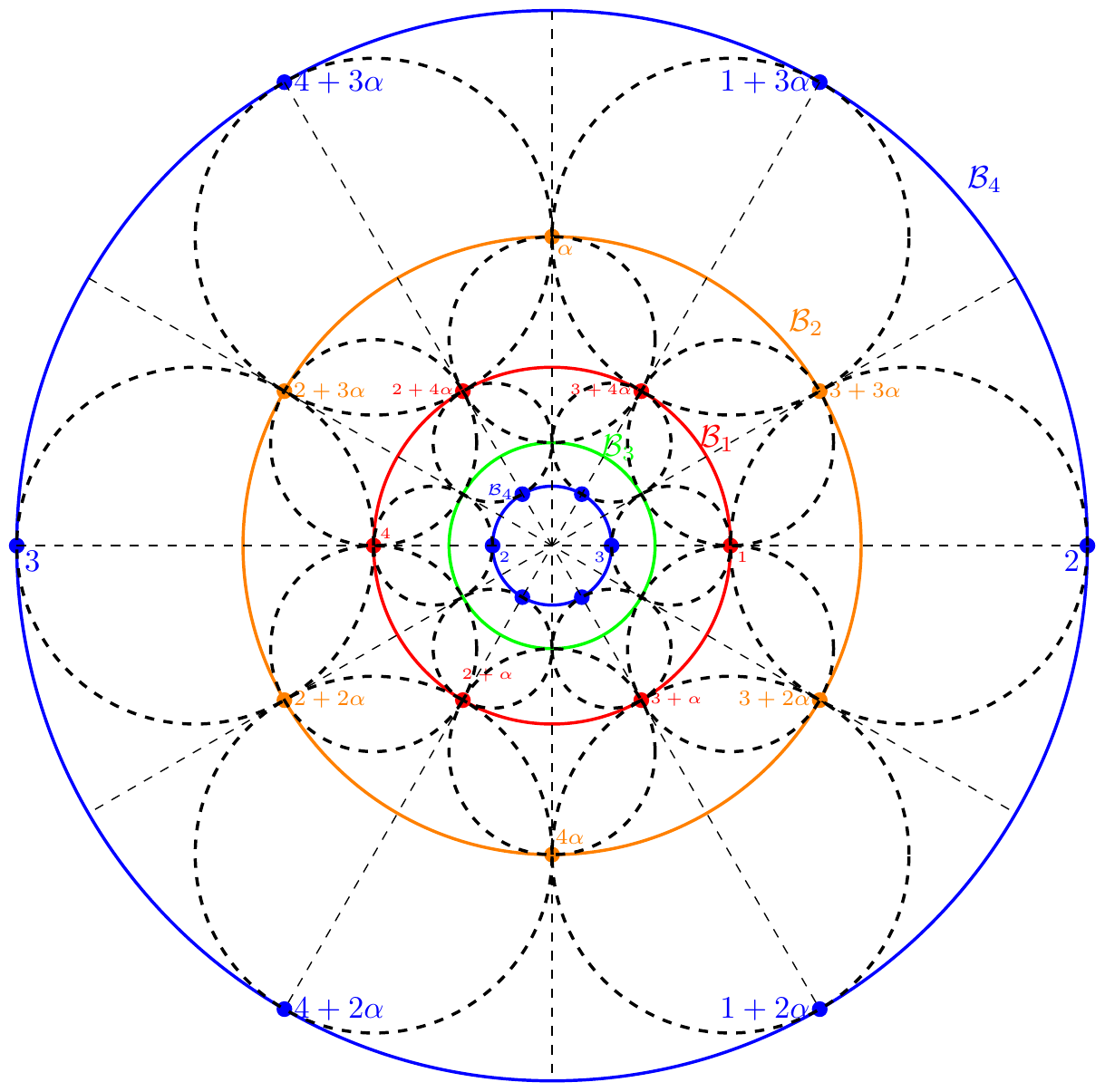}
  \caption{An example in $\bM(5)$.}
\end{figure}


\section{Some properties of concentric circles in $\bM(p^m)$}

  The example in the previous chapter suggests to look at pairs $\cB^1_{(0,a)}$ and  $\cB^1_{(0,b)}$.
  For a circle of the first type $\cB^1_{(s,c)}$, we will refer to $s$ as the \emph{center} of $\cB^1_{(s,c)}$. 
  Two circles of the first type are called \emph{concentric}, 
  if they are of the form $\cB^1_{(s,c)}$ and $\cB^1_{(s,c')}$ 
  for $s \in GF(p^{2m})$ and $c,c' \in GF(p^m) \setminus \{0\} $.
  Without loss of generality, we can always assume that two concentric circles have center $0$, 
  since the M\"obius transformation
    $$ \Phi \colon \bM(p^{m}) \rightarrow \bM(p^{m}), \ \Phi(z)=z-s $$
  maps any two concentric circles with center $s$ to two concentric circles with center $0$.
  In this chapter we henceforth consider two concentric circles with center $0$, 
  i.e.\ circles $\cB^1_{(0,a)}$ and $\cB^1_{(0,b)}$ for $a,b \in GF(p^m) \setminus\{0\}$. 
  For notational convenience, let us define
    $$ \cB_a:=\cB^1_{(0,a)}$$
  for all $a \in GF(p^m) \setminus\{0\}$.

  \begin{lemma}  \label{tangentConcentric}    There  are  $p^{2m}-1$
    circles tangential to $\cB_a$,  namely $(p^m+1)(p^m-2)$ circles of the
    first  type  and  $p^m+1$  circles  of  the  second  type.    In
    particular, $\cB^1_{(s,c)}$  is tangential to $\cB_a$  if and only
    if
    \begin{equation}\label{tang1}
      (c-a-s\overline{s})^2 = 4s\overline{s}a
    \end{equation}
  and $\cB^2_{(s,c)}$ is tangential to $\cB_a$ if and only if 
    \begin{equation}\label{tang2}
      c^2 = 4s\overline{s}a.
    \end{equation}
\end{lemma}

\begin{proof} 
  Let us start with circles of the first type given by $\cB^1_{(s,c)}: (z-s)(\overline{z}-\overline{s}) = c$ 
  for $s \in GF(p^{2m})$ and $c \in GF(p^m) \backslash \{0\}$. 
  Intersecting $\cB_a$ with $\cB^1_{(s,c)}$ leads to a quadratic equation in $z$, given by
    \begin{equation} \label{eq: intersect B_a with circle of first type}
      \overline{s}z^2 +(c-a-s\overline{s}) z + sa = 0.
    \end{equation}
  For $\cB^1_{(s,c)}$ being tangential to $\cB_a$, the discriminant needs to vanish, which leads to
    \begin{equation*}
      (c-a-s\overline{s})^2 = 4s\overline{s}a.
    \end{equation*}
  A standard counting argument shows that there are $(p^m+1)(p^m-2)$ 
  circles $\cB^1_{(s,c)}$ satisfying \eqref{tang1}.
  Similarly for circles of the second type.
\end{proof}

  Our main goal is to investigate Steiner chains in $\bM(p^m)$. 
  Starting with two concentric circles $\cB_a$ and $\cB_b$, 
  we need circles which are tangential to both of them.
  Those circles can easily be deduced using Lemma \ref{tangentConcentric}, 
  as we will see in the following.
  
\begin{corollary} \label{cor: common tangents} 
  If $\frac{b}{a}\neq 1$ is a square in $GF(p^m)$, 
  then the two circles $\cB_a$ and $\cB_b$ have exactly $2(p^m+1)$ common tangent circles.
  In particular, they are given by circles $\cB^1_{(s,c)}$ which satisfy
    \begin{equation} \label{commTang}
      \left( \frac{c-b-s\overline{s}}{c-a-s\overline{s}} \right) ^2 = \frac{b}{a}.
    \end{equation}
  If $\frac{b}{a }$ is a non-square in $GF(p^m)$, 
  then the circles $\cB_a$ and $\cB_b$ do not have any common tangent circles.
\end{corollary}

\begin{proof} 
  By Equation \eqref{tang1} the tangent circles of $\cB_a$ and $\cB_b$ are given by 
  $(c-a-s\overline{s})^2 = 4s\overline{s}a$ and 
  $(c-b-s\overline{s})^2 = 4s\overline{s}b$, respectively. 
  Combining these two equations, the conditions \eqref{commTang} and 
  $\frac{b}{a}$ being a square are immediate. 

  Let us now fix $c \in GF(p^m)\backslash\{0\}$ and look at
    $$ \mu: GF(p^{2m}) \rightarrow GF(p^m),\ s \mapsto \left( \frac{c-b-s\overline{s}}{c-a-s\overline{s}}\right)  - \sqrt{\frac{b}{a}}.$$
  For the zeros of $\mu(z)$, we may consider 
   $$ (c-b-s\overline{s})  - \sqrt{\frac{b}{a}}\ ({c-a-s\overline{s}}),$$
  since $(c-a-s\overline{s}) \neq 0$.
  This is a polynomial of degree $p^m+1$ and hence has at most $p^m+1$ zeros. 
  Note that by Lemma \ref{tangentConcentric}, we can exclude $s=0$. 
  When considering 
    $$ \mu^*: GF(p^{2m})\backslash \{0\} \rightarrow GF(p^m)\backslash \{0\},\ s \mapsto \left( \frac{c-b-s\overline{s}}{c-a-s\overline{s}}\right),$$
  by the pigeonhole principle every image is obtained $\frac{p^{2m}-1}{p^m-1}=p^m+1$ times. 
  Hence, $\mu$ has exactly $p^m+1$ zeros and (\ref{commTang}) has $2(p^m+1)$ solutions.

  Moreover, equation \eqref{tang2} directly implies that no common tangent circles of the second type are possible.
\end{proof}

  For $a \in GF(p^m)\setminus\{0\}$, let
    $$ \tau(a) := \{g \in \mathbb{B}: |\cB_a \cap g|=1 \} $$
  denote the set of all tangent circles of $\cB_a$ and
    $$ \tau(a,b) := \tau(a) \cap \tau(b) $$
  the set of all common tangent circles of $\cB_a$ and $\cB_b$. 
  
  Note that for every point $P \in \cB_a$, also $-P \in \cB_a$. 
  There is a circle of the second type through $P$, $-P$ and $0$, 
  which could be thought of as the line through the center of $\cB_a$ and two of its points. 
  Because of that, we will refer to $-P$ as the \emph{opposite point of $P$ on $\cB_a$}. 
  Moreover, $P$ and $-P$ are called \emph{opposite points}.

  For any $P \in \cB^1_{(s,c)}$ we have $-P \in \cB^1_{(-s,c)}$
  and we call $\cB^1_{(-s,c)}$ the \emph{opposite} circle of $\cB^1_{(s,c)}$.
  Let $\frac{b}{a}$ be a square in $GF(p^m)$.
\begin{lemma}
  The circles $\cB^1_{(s_1,c_1)}$ and $\cB^1_{(s_2,c_2)}$ in $\tau(a,b)$
  are tangential to the same point $P \in \cB_a$ or to opposite points $P$ and $-P$ on $\cB_a$, respectively, if and only if
    \begin{equation} \label{eq: tangent to same or opposite points}
      s_1 \overline{s_2} \in GF(p^m) \setminus\{0\}.
    \end{equation}
  In particular, $P$ is given by
    \begin{equation} \label{eq: P}
      \frac{P^2}{a} = \frac{s_1}{\overline{s_1}} =\frac{s_2}{\overline{s_2}}.
    \end{equation}
\end{lemma}

\begin{proof}
  By Lemma \ref{tangentConcentric}, we have $(c-a-s\overline{s})^2 = 4s\overline{s}a$,
  i.e.\ the discriminant in \eqref{eq: intersect B_a with circle of first type} vanishes, 
  which directly leads to $ z = \frac{-c+a+s\overline{s}}{2 \overline{s}}$ and furthermore to
    \begin{equation} \label{eq: tangent point}
      \frac{z^2}{a} = \frac{s}{\overline{s}}.
    \end{equation}
  Now, let $\cB^1_{(s_1,c_1)}$ and $\cB^1_{(s_2,c_2)}$ in $\tau(a,b)$ be tangential to the same point $P \in \cB_a$ 
  or to opposite points $P$ and $-P$ on $\cB_a$. 
  By \eqref{eq: tangent point}, this gives
    $$ \frac{P^2}{a} = \frac{s_1}{\overline{s_1}} = \frac{s_2}{\overline{s_2}} $$
  and hence $ s_1 \overline{s_2} = \overline{s_1} s_2 = \overline{s_1 \overline{s_2}}$,
  which means $s_1 \overline{s_2} \in GF(p^m) \setminus\{0\}$.

  Conversely, let $\cB^1_{(s_1,c_1)}$ and $\cB^1_{(s_2,c_2)}$ in $\tau(a,b)$ be tangential to $\cB_a$ in $P$ and $Q$, respectively. 
  Then $s_1 \overline{s_2} \in GF(p^m) \setminus\{0\}$ and \eqref{eq: tangent point} imply $P^2 = Q^2$.
\end{proof}

\begin{lemma} \label{lemma: tangent points}
  Let $\mu^2=\frac{b}{a}\neq 1$ be a square in $GF(p^m)$. 
  Consider $\cB^1_{(s,c)} \in \tau(a,b)$ and $P \in \cB_a \cap \cB^1_{(s,c)}$. 
  Then $\cB_b \cap \cB^1_{(s,c)}$ is either $\mu P$ or $-\mu P$.
\end{lemma}

\begin{proof} 
  Let $P \in \cB_a \cap \cB^1_{(s,c)}$, i.e.\ $ a -s\overline{P}-\overline{s}P+s\overline{s}=c$.
  Reformulating gives
    $$ (-s\overline{P}-\overline{s}P)=(c-a-s\overline{s}).$$
  Since $\cB^1_{(s,c)}$ is a common tangent circle of $\cB_a$ and $\cB_b$, we can use \eqref{commTang} and obtain
    $$ (-s\overline{P}-\overline{s}P)^2=\frac{a}{b} (c-b-s\overline{s})^2,$$
  which leads to
    $$ (-s\mu \overline{P}-\overline{s}\mu P)^2= (c-b-s\overline{s})^2.$$
  Hence, by the same reformulation as above, 
  either $\mu P \in \cB^1_{(s,c)}$ or $-\mu P \in \cB^1_{(s,c)}$. 
  Since $\mu P$ and $-\mu P$ are both on $\cB_b$, we are done.
\end{proof}

Combining the results above, we can now state and prove the following result about common tangent circles of two concentric circles.

\begin{theorem} \label{th: exactly two tangent circles}
  Let $\mu^2=\frac{b}{a}\neq 1$ be a square in $GF(p^m)$. 
  For every point $P$ on $\cB_a$ there are exactly two circles $g,h \in \tau(a,b)$ tangential to $\cB_a$ in $P$. 
  Moreover, $g$ is tangential to $\cB_b$ in $\mu P$ and $h$ is tangential to $\cB_b$ in $-\mu P$.
\end{theorem}

\begin{proof} 
  By Lemma \ref{lemma: tangent points} we know that if $\cB^1_{(s,c)} \in \tau(a,b)$ is tangential to $\cB_a$ in $P$, 
  then $\cB^1_{(s,c)}$ is tangential to $\cB_b$ in $\mu P$ or in $-\mu P$. 
  Hence, for every point $P$ on $\cB_a$ there are at most two circles in $\tau(a,b)$ which are tangential to $\cB_a$ in $P$ 
  by the uniqueness of tangent circles, i.e.\ Axiom (M2) for M\"obius planes. 
  There are $p^m+1$ points on $\cB_a$ and by Corollary \ref{cor: common tangents},
  there are exactly $2(p^m+1)$ circles in $\tau(a,b)$. 
  Hence, for every point $P$ on $\cB_a$, there are exactly two circles in $\tau(a,b)$ tangential to $\cB_a$ in $P$.
\end{proof}

\begin{corollary} \label{cor: two groups of tangent circles} 
  Let $\mu^2=b\neq 1$ be a square in $GF(p^m)$. 
  The $2(p^m+1)$ common tangent circles of $\cB_1$ and $\cB_b$ are given by
    $$\{ \cB^1_{(sP,c)}| P\in \cB_1 \} \ \bigcup  \ \{ \cB^1_{(s'P,c')}| P\in \cB_1 \}.$$
  In particular, $P,\mu P \in \cB^1_{(sP,c)}$ for all $P \in \cB_1$ with 
    \begin{align} \label{eq: s and c}
      s = \frac{1+\mu}{2},\qquad c=\left(\frac{1-\mu}{2}\right)^2
    \end{align}
  and $P,-\mu P \in \cB^1_{(s'P,c')}$ for all $P \in \cB_1$ with
    \begin{align} \label{eq: s' and c'}
      s' = \frac{1-\mu}{2},\qquad c'=\left(\frac{1+\mu}{2}\right)^2.
    \end{align}
\end{corollary}

\begin{proof}  
  By Theorem \ref{th: exactly two tangent circles}, there is a $\cB^1_{(s,c)} \in \tau(1,b)$ such that $1,\mu \in \cB^1_{(s,c)}$. 
  Then $P, \mu P \in \cB^1_{(sP,c)}$ for any point $P$ on $\cB_1$, which can easily be checked. 
  Moreover, $\cB^1_{(sP,c)}$ lies indeed in $\tau(1,b)$. 
  Using $P\overline{P}=1$ and $\cB^1_{(s,c)} \in \tau(1,b)$,
  this follows directly by the condition obtained in Corollary \ref{cor: common tangents}, namely
    $$\left( \frac{c-b-(sP)(\overline{sP})}{c-1-(sP)(\overline{sP})}\right) ^2 = b.$$
  The parameters $s,c$ and $s',c'$ can easily be deduced by using $1,\mu \in \cB^1_{(s,c)}$ and $1,-\mu \in \cB^1_{(s',c')}$.
\end{proof}
The following figure visualizes our results about common tangent circles of $\cB_a$ and $\cB_b$.

\begin{figure}[h]\label{fig: Opposite circles}
  \centering
  \includegraphics[width=7cm]{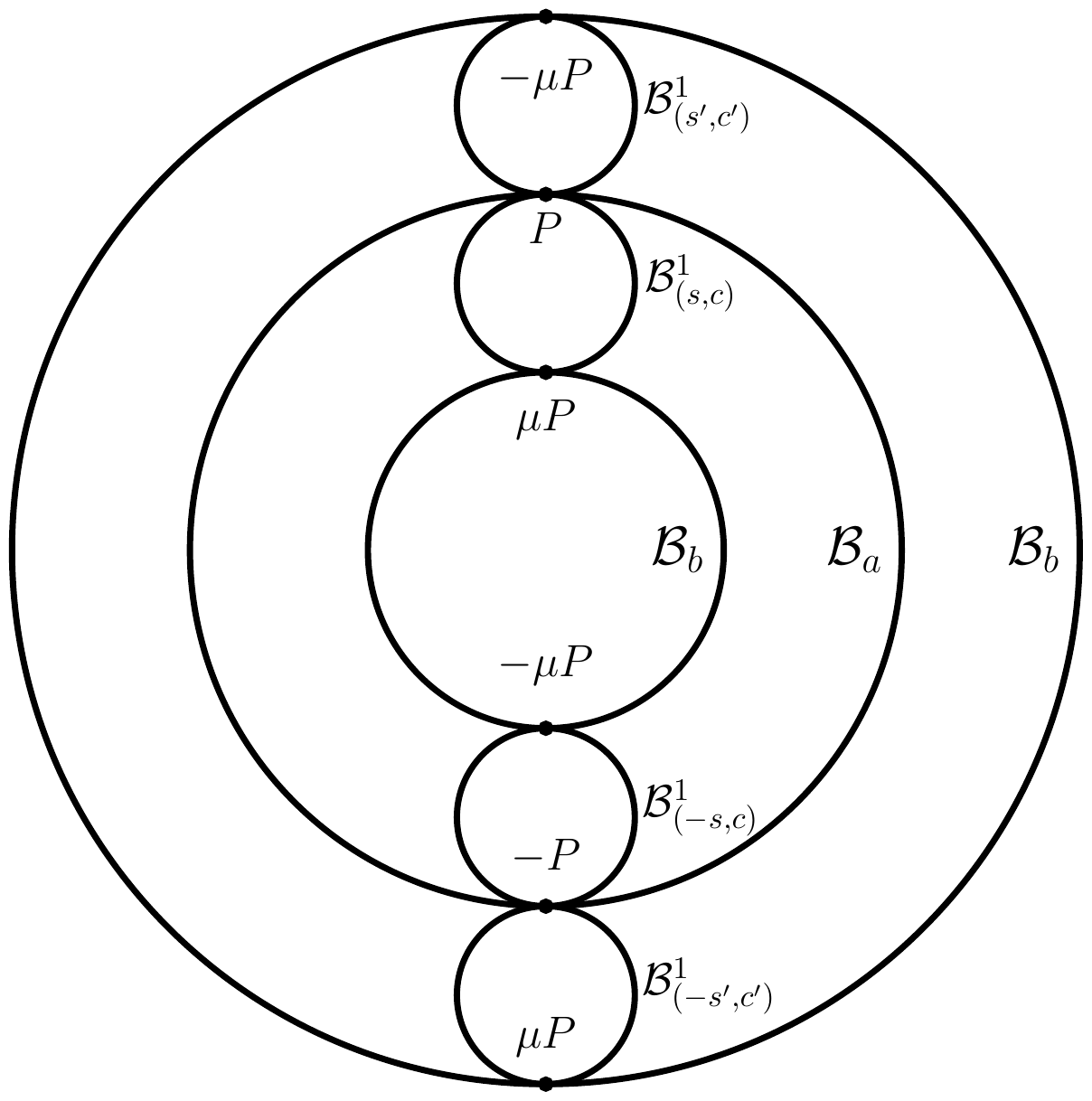}
  \caption{Common tangent circles of $\cB_a$ and $\cB_b$}
\end{figure}


\section{Steiner chains in $\bM(p^m)$}

In this section we construct Steiner chains using two concentric circles $\cB_1$ and $\cB_b$, 
for $b$ a square in $GF(p^m)$, and their common tangent circles.

Note that the following discussion already covers all cases for two concentric circles $\cB_a$ and $\cB_b$.
To see this, look at the M\"obius transformation
$$ \Phi(z) = \frac{z}{\eta}$$
for any $\eta$ in $GF(p^{2m})$ such that $\eta \overline{\eta} = a$. 
Note that such an $\eta$ always exists by the properties of the norm map.
Then $\Phi$ maps the circles $\cB_a$ and $\cB_b$ to $\cB_1$ and $\cB_{\frac{b}{a}}$, respectively.

\subsection{Steiners Theorem in $\mathbb{M}(p^m)$}

We are interested in constructing Steiner chains. 
The next result shows that  
if we find two common tangent circles of $\cB_1$ and $\cB_b$, which are also tangential to each other, 
a Steiner chain can be constructed.

\begin{lemma} \label{lemma: tangent circles are tangential}
  Let $b\neq 1$ be a square in $GF(p^m)$ and 
  consider the set $\{ \cB^1_{(sP,c)} \}$ of circles in $\tau(1,b)$ 
  for $P$ ranging through $\cB_1$.
  If
    $$ |\cB^1_{(s,c)} \cap \cB^1_{(sP,c)}| = 1,$$
  for some $P \in \cB_1$, then
    $$ |\cB^1_{(sQ,c)} \cap \cB^1_{(sPQ,c)}| = 1$$
  for all $Q \in \cB_1$.
  Moreover, $|\cB^1_{(s,c)} \cap \cB^1_{(sP,c)}| = 1$ if and only if
    \begin{equation} \label{eq: tangential circles are tangential to each other}
      s(P-1)\overline{s(P-1)} = 4c. 
    \end{equation}
\end{lemma}

\begin{proof} 
  By Lemma \ref{cor: two groups of tangent circles}, 
  we may consider two circles in $\tau(1,b)$ given by  $\cB^1_{(s,c)}$ and $\cB^1_{(sP,c)}$ for some $P \in \cB_1$. 
  Let us apply the M\"obius transformation $\Phi(z) = z-s$, 
  which maps $\cB^1_{(s,c)}$ to $\cB^1_{(0,c)}$ and $\cB^1_{(sP,c)}$ to $\cB^1_{(s(P-1),c)}$.
  We need $\cB^1_{(0,c)}$ and $\cB^1_{(s(P-1),c)}$ being tangential. 
  Hence, the discriminant of the quadratic equation obtained by intersecting $\cB^1_{(0,c)}$ and $\cB^1_{(s(P-1),c)}$ needs to vanish, 
  which is exactly \eqref{eq: tangential circles are tangential to each other}.

  Assume $|\cB^1_{(s,c)} \cap \cB^1_{(sP,c)}| = 1$, i.e.\ \eqref{eq: tangential circles are tangential to each other} is satisfied. 
  Let us apply the M\"obius transformation $\Phi(z)=z-sQ$ to $\cB^1_{(sQ,c)}$ and $\cB^1_{(sPQ,c)}$ 
  for some fixed $Q \in \cB_1$. 
  We consider the circles $\cB^1_{(0,c)}$ and $\cB^1_{(sQ(P-1),c)}$. 
  Again, the discriminant of the quadratic equation obtained by intersecting those two circles needs to vanish, 
  i.e.\ we obtain the condition
    $$ sQ(P-1)\overline{sQ(P-1)} = 4c.$$
  Since $Q\overline{Q}=1$, the above equation is satisfied if and only if 
  \eqref{eq: tangential circles are tangential to each other} is satisfied, which proves the claim.
\end{proof}

Now we are able to state and prove our finite version of Steiner's Theorem.
A Steiner chain of \emph{length $k$} for $\cB_1$ and $\cB_b$, $k \geq 3$, is a chain of 
$k$ different circles $T_1,\ldots,T_k$ in $\tau(1,b)$ such that $|T_i \cap T_{i+1}|=1$ for $i =1,\ldots,k-1$
and $ |T_k \cap T_1| = 1$.

\begin{theorem} \label{th: finite Steiner}
  Consider two different concentric circles $\cB_1$ and $\cB_b$. 
  Assume we can find two circles $\cB^1_{(sP,c)}$ and $\cB^1_{(sQ,c)}$ in $\tau(1,b)$ for some $P$ and $Q$ on $\cB_1$, 
  which are tangential to each other.
  Then a Steiner chain of length $k$ for some $k \in \{3,\ldots,p^m+1\}$ can be constructed.
  Moreover, a Steiner chain of the same length $k$ can be constructed starting with any circle $B$ in 
    $$\{ \cB^1_{(sR,c)}| \ R \in \cB_1 \}.$$ 
\end{theorem}

\begin{proof}
  Without loss of generality, we assume the existence of two circles $\cB^1_{(s,c)}$ and $\cB^1_{(sP,c)}$ in $\tau(1,b)$, 
  which are tangential to each other. 
  By choosing $P = Q$ in Lemma \ref{lemma: tangent circles are tangential}, 
  we know that $\cB^1_{(sP,c)}$ and $\cB^1_{(sP^2,c)}$ in $\tau(1,b)$ are tangential to each other as well.
  There exists a minimal $k \in \{3, \ldots, p^m+1 \}$ such that $ P^k = 1$. Hence,
    $$ |\cB^1_{(s,c)} \cap \cB^1_{(sP,c)}| = |\cB^1_{(sP,c)} \cap \cB^1_{(sP^2,c)}| = \ldots = |\cB^1_{(sP^{k-1},c)} \cap \cB^1_{(s,c)}| = 1,$$
  which gives a Steiner chain of length $k$.
  
  Now, consider an arbitrary circle $\cB^1_{(sR,c)}$ in $\tau(1,b)$ for $R$ on $\cB_1$.
  Again by Lemma \ref{lemma: tangent circles are tangential}, we get
    $$ |\cB^1_{(sR,c)} \cap \cB^1_{(sPR,c)}| = |\cB^1_{(sPR,c)} \cap \cB^1_{(sP^2R,c)}| = \ldots = |\cB^1_{(sP^{k-1}R,c)} \cap \cB^1_{(sR,c)}| = 1$$
  and the chain closes again after $k$ steps.
\end{proof}

We will call a Steiner chain for $\cB_a$ and $\cB_b$ \emph{proper}, 
if all the common tangent circles in the chain are taken from the same group of common tangent circles,
i.e.\ they are either all taken from $\{ \cB^1_{(sP,c)}| P\in \cB_1 \}$ or from $\{ \cB^1_{(s'P,c')}| P\in \cB_1 \}$,
where the parameters are given by \eqref{eq: s and c} and \eqref{eq: s' and c'}, respectively.
Otherwise, the Steiner chain is called \emph{degenerated}.
Proper Steiner chains are  equivalently characterized by the property,
that the total  number of contact points equals 3  times the length of
the chain, or that no point is contact point of more than two circles.

\begin{figure}[h!]\label{fig: proper and degenerated Steiner chain}
  \centering
  \includegraphics[width=6cm]{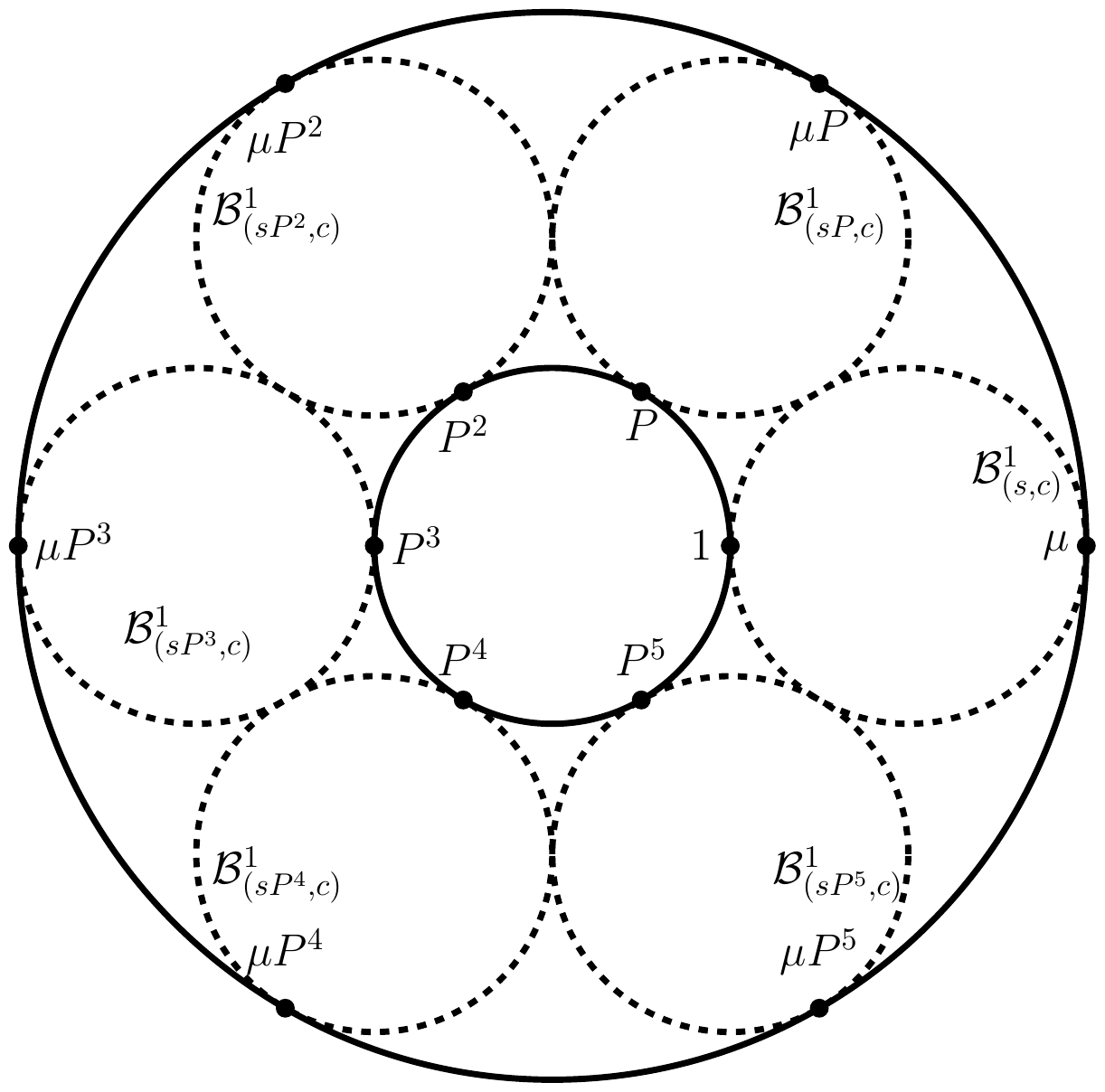} \qquad \includegraphics[width=6cm]{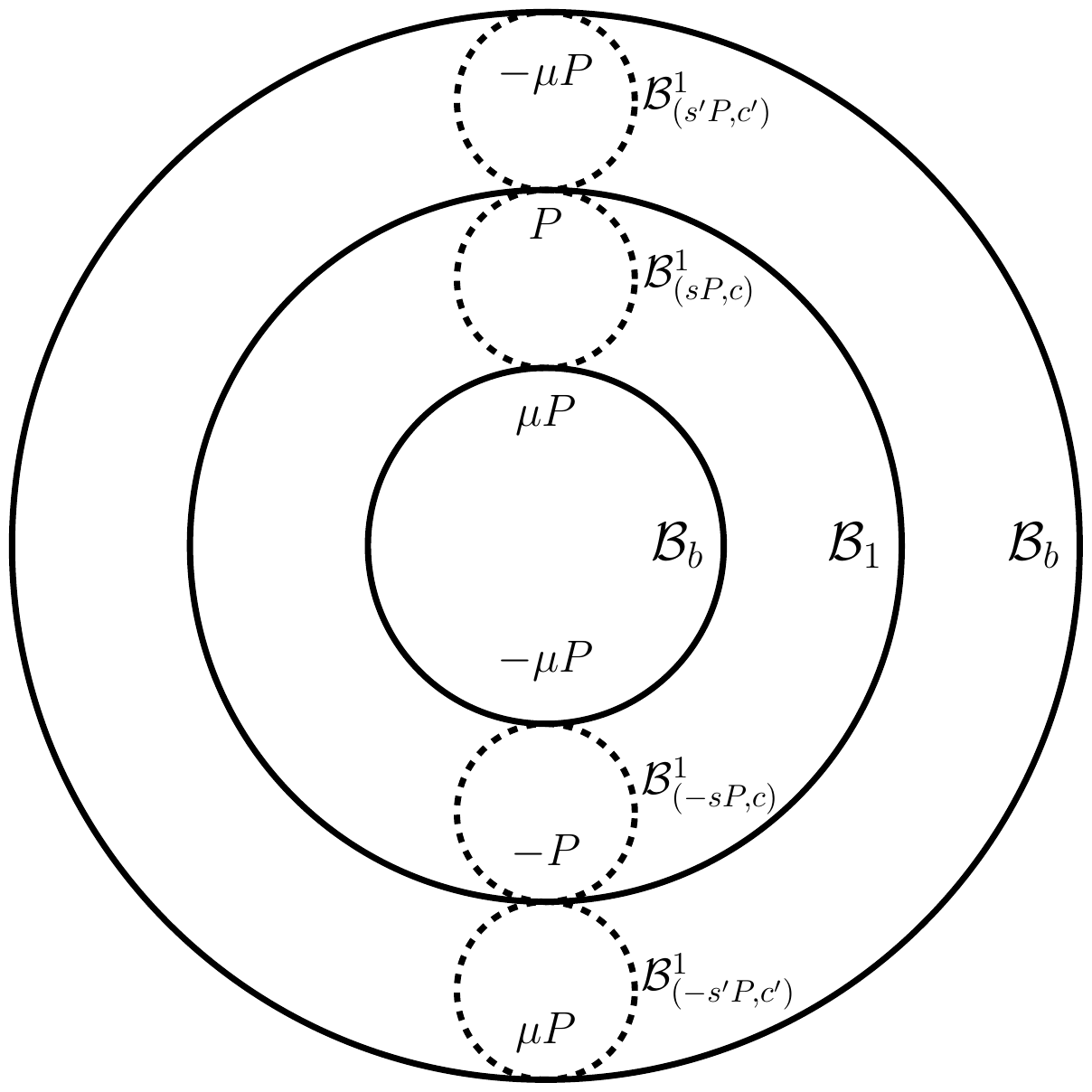}
  \caption{On the left, we have a \emph{proper} Steiner chain given by $ \cB^1_{(s,c)} \to \cB^1_{(sP,c)} \to \cB^1_{(sP^2,c)} \to \cB^1_{(sP^3,c)} \to \cB^1_{(sP^4,c)} \to \cB^1_{(sP^5,c)} \to \cB^1_{(s,c)}$ and on the right, we have a \emph{degenerate} Steiner chain given by $\cB^1_{(sP,c)} \to \cB^1_{(s'P,c)} \to \cB^1_{(-sP,c)}\to \cB^1_{(-s'P,c)} \to \cB^1_{(sP,c)}$.}
\end{figure}

\subsection{Existence and length of Steiner chains}

In what follows, we are mainly interested in the existence of proper Steiner chains 
as well as their possible lengths. 
As discussed in the previous section, 
the existence of degenerate Steiner chains for $\cB_1$ and $\cB_b$ is guaranteed as long as $b$ is a square in $GF(p^m)$.
For proper Steiner chains, however, we obtain the following criterion.

\begin{lemma} \label{lemma: P1 and P2}
  Let $\mu^2=b\neq 1$ be a square in $GF(p^m)$. 
  A proper Steiner chain for $\cB_1$ and $\cB_b$ can be constructed
  using circles in $\{ \cB^1_{(sP,c)}| \ P \in \cB_1 \}$
  for $s = \frac{1+\mu}{2}$ and $c=\left(\frac{1-\mu}{2}\right)^2$
  if and only if $- \mu$ is a non-square in $GF(p^m)$.
  
  In particular, for $-\mu $ a non-square in $GF(p^m)$,
  a proper Steiner chain for $\cB_1$ and $\cB_b$ can be constructed starting with
  $\cB^1_{(s,c)}$ and $\cB^1_{(sP_1,c)}$ and closing up with $\cB^1_{(sP_2,c)}$ for
    \begin{align} \label{eq: Tangential point P1}
      P_1 = \frac{(-\mu^2+6\mu-1) + 4(\mu-1)\sqrt{-\mu}}{(1+\mu)^2}
    \end{align}
  and 
    \begin{align} \label{eq: Tangential point P2}
      P_2 = \frac{(-\mu^2+6\mu-1) - 4(\mu-1)\sqrt{-\mu}}{(1+\mu)^2}.
    \end{align}
\end{lemma} 

\begin{proof}
  By Theorem \ref{th: finite Steiner}, we can start with $\cB^1_{(s,c)}$.
  This circle is tangential to $\cB^1_{(sP,c)}$ if and only if $s(P-1)\overline{s(P-1)} = 4c$, 
  as seen in Lemma \ref{lemma: tangent circles are tangential}.
  Using $P\overline{P}=1$, this gives a quadratic equation in $P$.
  Solving the quadratic equation and using $s = \frac{1+\mu}{2}$ and $c=\left(\frac{1-\mu}{2}\right)^2$
  leads exactly to the two solutions $P_1$ and $P_2$ above.
  
  For $-\mu$ a square in $GF(p^m)$,
  we obtain $P_1$ and $P_2$ both in $GF(p^m)$.
  Since $1$ and $-1$ are the only elements in $GF(p^m)$ on $\cB_1$,
  we obtain here the opposite circle of $\cB^1_{(s,c)}$,
  which results into a degenerate Steiner chain.
\end{proof}

Before we state our general result about the length of a Steiner chain,
we discuss some specific cases for $\bM(p)$, $p$ an odd prime, in detail.

First, let us examine a criterion for Steiner chains of length $3$.

\begin{lemma} \label{lemma: 3 chain}
  Let $b=\mu^2$ for $\mu$ in $GF(p)$.
  A proper Steiner chain of length $3$ can be constructed if and only if 
    $$ \mu = 7 + 4 \sqrt{3} \in GF(p) $$
  and $-\mu$ a non-square in $GF(p)$.
In particular, if this condition is
satisfied, $\cB_1$ and $\cB_b$ carry a chain of length 3.
\end{lemma}

\begin{proof}
  By Lemma \ref{lemma: P1 and P2}, 
  we know that $-\mu$ necessarily needs to be a non-square in $GF(p)$.
  Moreover, if $-\mu$ is a square in $GF(p)$, 
  we can find two circles in $\tau(1,b)$ which are tangential to $\cB^1_{(s,c)}$,
  namely $\cB^1_{(sP_1,c)}$ and $\cB^1_{(sP_2,c)}$, for $P_1$ and $P_2$ given by
  \eqref{eq: Tangential point P1} and \eqref{eq: Tangential point P2}, respectively.
  For a Steiner chain of length $3$, also $\cB^1_{(sP_1,c)}$ and $\cB^1_{(sP_2,c)}$
  need to be tangential.
 
  \begin{figure}[h!]\label{fig: length three}
    \centering
      \includegraphics[width=7cm]{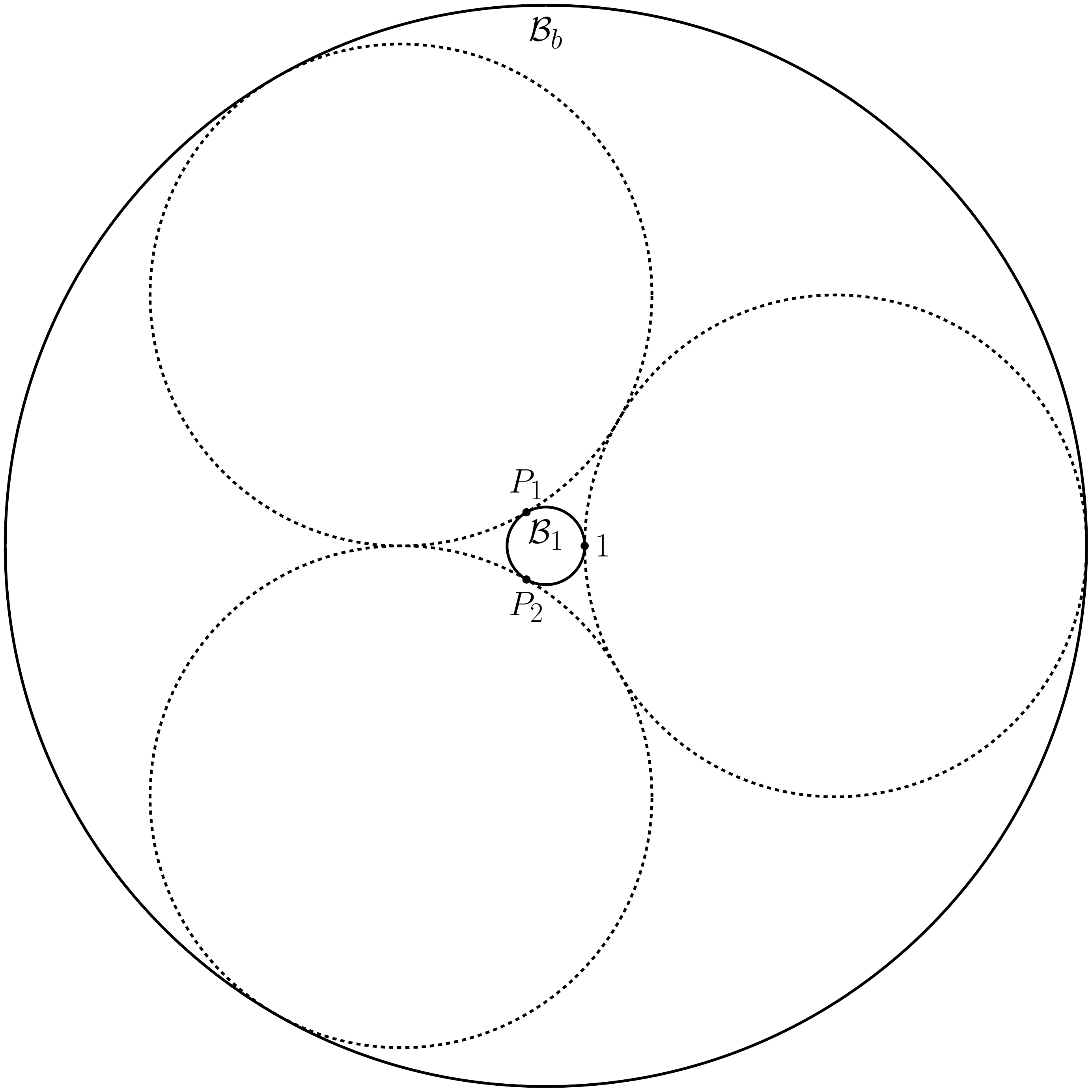}
        \caption{A Steiner chain of length $3$.}
  \end{figure} 
  
  So, we need $P_1^3=1$ or similar $P_1^2 = P_2.$
  Solving this equation for $\mu$ leads to 
    $$ \mu = 7 + 4 \sqrt{3}.$$
  This is only possible for $3$ a square in $GF(p)$.
\end{proof}

Hence, we are interested in a condition on $3$ being a square in $GF(p)$. 
It is well-known (see for example \cite{MR2445243}) that
$$ 3 \text{ is a square in } GF(p) \Leftrightarrow p \equiv \pm 1 \mod \ 12.$$
This already excludes the existence of a Steiner chain of length $3$ in $\bM(5)$.
Indeed, we have seen in Section \ref{section: Example} 
that only Steiner chains of length $6$ occur for two concentric circles in $\bM(5)$.

In $\bM(11)$, however, we can find Steiner chains of length $3$.
For this, calculate
$$ \mu = 7 + 4 \sqrt{3} = \begin{cases} 7+9=5 \mod \ 11 \\ 7+2=9 \mod \ 11. \end{cases} $$
Moreover, $-5=6$ as well as $-9=2$ are non-squares in $GF(11)$. 
So, Steiner chains of length $3$ can be constructed for $\cB_1$ and $\cB_3$ as well as for
$\cB_1$ and $\cB_4$.


Now, let us have a look at Steiner chains of length $4$.

\begin{lemma} \label{lemma: 4 chain}
  Let $b = \mu^2$ for $\mu$ in $GF(p)$.
  A proper Steiner chain of length $4$ can be constructed if and only if 
    $$ \mu = 3 + 2 \sqrt{2} \in GF(p) $$
  and $-\mu$ a non-square in $GF(p)$. 
In particular, if this condition is
satisfied, $\cB_1$ and $\cB_b$ carry a chain of length 4.
\end{lemma}

\begin{proof}
  Again, by Lemma \ref{lemma: P1 and P2}, 
  $-\mu$ needs to be a non-square in $GF(p)$.
  Moreover, if $-\mu$ is a square in $GF(p)$, 
  we have two circles $\cB^1_{(sP_1,c)}$ and $\cB^1_{(sP_2,c)}$ in $\tau(1,b)$, 
  which are tangential to $\cB^1_{(s,c)}$ with $P_1$ and $P_2$ given by
  \eqref{eq: Tangential point P1} and \eqref{eq: Tangential point P2}, respectively.
  For a Steiner chain of length $4$, we need $P_1^4=1$ or similar, $P_1 = -P_2$.
  \begin{figure}[h]\label{four}
    \centering
      \includegraphics[width=5cm]{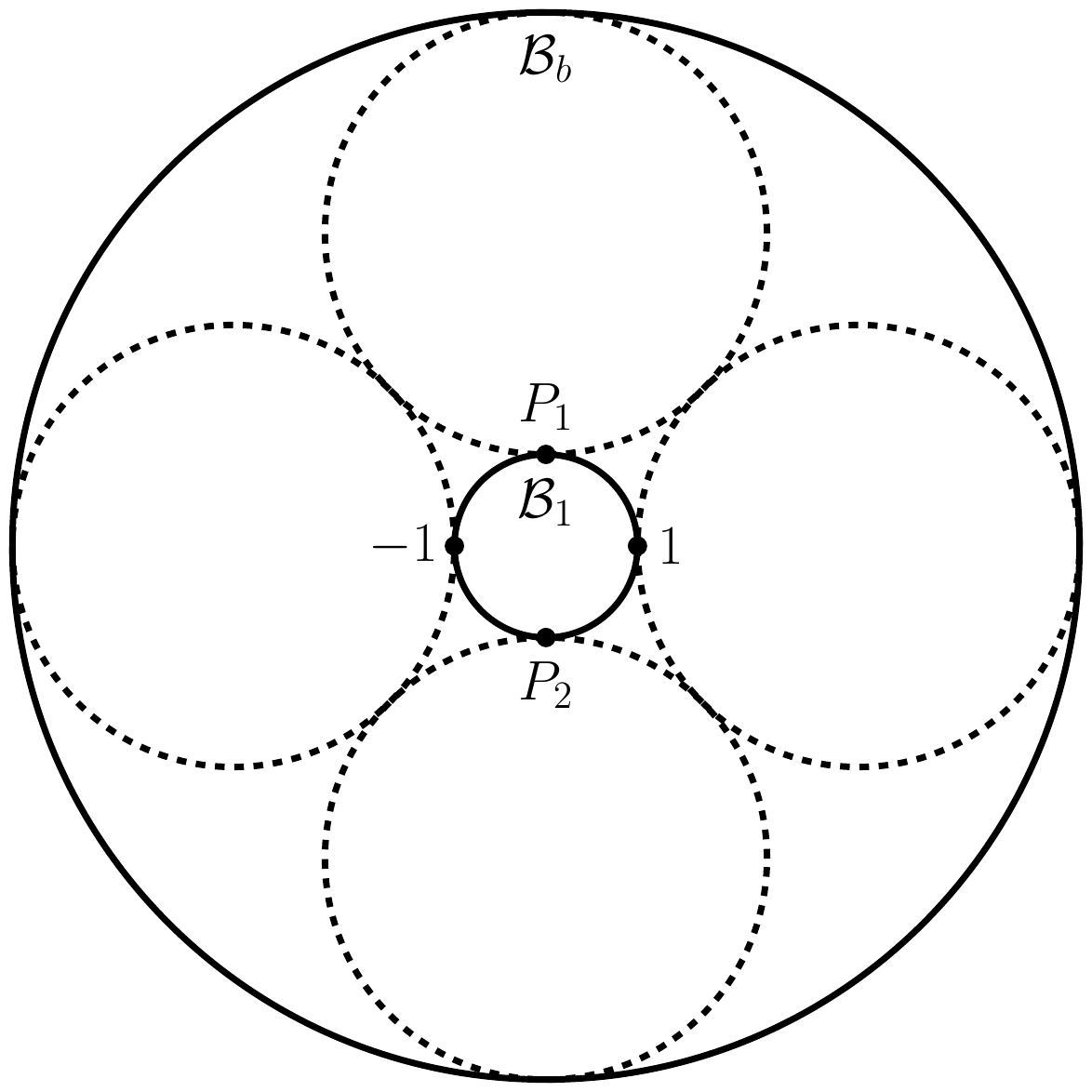}
        \caption{A Steiner chain of length $4$.}
  \end{figure} 
  Solving this equation for $\mu$ gives 
    $$ \mu = 3 + 2 \sqrt{2}.$$
  This is only possible for $2$ a square in $GF(p)$.
\end{proof}

We are interested in a condition on $2$ being a square in $GF(p)$. 
Again from \cite{MR2445243}, we have that
  $$ 2 \text{ is a square in } GF(p) \Leftrightarrow p \equiv \pm 1 \mod \ 8.$$

Let us have a look at $\bM(7)$ and calculate
  $$ \mu = 3 + 2 \sqrt{2} = \begin{cases} 3+6=2 \mod \ 7 \\ 3+1=4 \mod \ 7. \end{cases} $$
Moreover, $-2=5$ as well as $-4=3$ are non-squares in $GF(7)$. 
So, Steiner chains of length $4$ can be constructed for $\cB_1$ and $\cB_2$ as well as
$\cB_1$ and $\cB_4$.
Note that $1$, $2$ and $4$ are the only squares in $GF(7)$, 
hence, only Steiner chains of length $4$ can be constructed using $\cB_1$ and $\cB_b$.
Moreover, $2$ is not a square in $GF(11)$, so in $\bM(11)$, there are no Steiner chains of length $4$.

Similarly, we can calculate a criterion for Steiner chains of length $6$.

\begin{lemma} \label{lemma: 6 chain}
  Let $b = \mu^2$ for $\mu$ in $GF(p)$.
  A proper Steiner chain of length $6$ can be constructed if and only if 
    $$ \mu = 3 \text{ or } \frac{1}{3} $$
  and $-\mu$ a non-square in $GF(p)$. 
In particular, if this condition is
satisfied, $\cB_1$ and $\cB_b$ carry a chain of length 6.
\end{lemma}

This criterion is different from the criterion for Steiner chains of length $3$ and $4$,
since no square roots appear in the expression for $\mu$ above. 

By Steiner's Theorem it follows directly that the existence of a Steiner 
chain of length $6$ in $GF(p)$ implies $6|(p+1)$. 
Indeed, the condition $6|(p+1)$ is actually equivalent with $-3$ being a non-square in $GF(p)$,
which can be seen by number theoretic calculations only.

Note that in $\bM(5)$, this gives $3^2=4=(\frac{1}{3})^2$, 
i.e.\ only for $\cB_1$ and $\cB_4$ a Steiner chain of length $6$ can be constructed.
Compare these results also to Section \ref{section: Example}.

Now, let us have a look at the conditions for Steiner chains of length $5$ and $8$.
The expression for $\mu$ becomes more and more difficult,
since equations of higher order need to be solved.
In particular, for Steiner chains of length $5$, we need to solve $P_1^3=P_2^2$ and
for Steiner chains of length $8$, we need to solve $P_1^2=-P_2^2$.

\begin{lemma} \label{lemma: 5 chain}
  Let $b = \mu^2$ for $\mu$ in $GF(p)$.
  A proper Steiner chain of length $5$ can be constructed if and only if 
    $$ \mu = 11 - 4 \sqrt{5} + 2 \sqrt{50 - 22 \sqrt{5}} \in GF(p) $$
  and $-\mu$ a non-square in $GF(p)$. 
In particular, if this condition is
satisfied, $\cB_1$ and $\cB_b$ carry a chain of length $5$.
\end{lemma}

We know that
  $$ 5 \text{ is a square in } GF(p) \Leftrightarrow p \equiv \pm 1 \mod \ 5 $$
which gives a necessary, but not sufficient condition for the existence of Steiner chains of length $5$.

\begin{lemma} \label{lemma: 8 chain}
  Let $b = \mu^2$ for $\mu$ in $GF(p)$.
  A proper Steiner chain of length $8$ can be constructed if and only if 
    $$ \mu = 7 - 4 \sqrt{2} + 2 \sqrt{2 (10 - 7 \sqrt{2})} \in GF(p) $$
  and $-\mu$ a non-square in $GF(p)$. 
In particular, if this condition is
satisfied, $\cB_1$ and $\cB_b$ carry a chain of length $8$.
\end{lemma}

Note that $2$ needs necessarily to be a square, similar to Steiner chains of length $4$. 
This is not surprising, since $8$ is a multiple of $4$.

Now we are ready to give a condition for Steiner chains of length $k$,
for $k \geq 3$.

\begin{theorem}\label{concentric}
  Let $b = \mu^2\neq 1$, for $\mu$ in $GF(p^m)$.
  A proper Steiner chain of length $k$ can be constructed if and only if the following conditions are satisfied.
    \begin{enumerate}
      \item $-\mu$ is a non-square in $GF(p^m)$,
      \item $\mu$ solves $ P_1^k=1$ for $P_1$ given by 
        \begin{align}
          P_1 = \frac{-\mu^2+6\mu-1 + 4(\mu-1)\sqrt{-\mu}}{(1+\mu)^2}
        \end{align}
            but $P_1^l \neq 1$ for all $1 \leq l \leq k-1$.
    \end{enumerate}
    In  particular,  if these  conditions  are  satisfied, $\cB_1$  and
    $\cB_b$ carry a chain of length $k$.
\end{theorem}

\begin{proof}
  Let us assume the existence of a Steiner chain of length $k$.
  By Lemma \ref{lemma: P1 and P2}, $-\mu$ needs to be a non-square in $GF(p^m)$ 
  to obtain two circles $\cB^1_{(sP_1,c)}$ and $\cB^1_{(sP_2,c)}$, which are both tangential to $\cB^1_{(s,c)}$.
  Again by Theorem \ref{th: finite Steiner}, 
  we know that starting with two such circles $\cB^1_{(sP_1,c)}$ and $\cB^1_{(s,c)}$ in $\tau(1,b)$,
  which are mutually tangential, we end up with a proper Steiner chain.
  Moreover, the length of the Steiner chain is then given by the smallest integer $k$, 
  such that $P_1^k=1$, i.e.\ we are back at the starting point.
  
  Now let us assume that the above three conditions are satisfied.
  Since $-\mu$ is a non-square, we can apply Lemma \ref{lemma: P1 and P2} to obtain a proper Steiner chain.
  Since $k$ is by assumption such that  $P_1^k=1$ but $P_1^l \neq 1$ for all $1 \leq l \leq k-1$, 
  the length of the proper Steiner chain is indeed $k$.
\end{proof}


\section{Generalization}
\subsection{A M\"obius invariant for pairs of circles}
In the Euclidean plane two non-intersecting circles form a 
capacitor. The capacitance is a conformal invariant and therefore
in particular invariant under M\"obius transformations. 
Here we present a discrete analogue of this fact 
which will be used later to decide whether any two
non-intersecting circles carry a proper Steiner chain of
length $k$.

The \emph{capacitance} associates a number in $GF(p^m)$ to any pair of circles
in $\mathbb M(p^m)$ and is defined as follows: 
\begin{eqnarray*}
\operatorname{cap}(\mathcal{B}^1_{(s_1,c_1)},\mathcal{B}^1_{(s_2,c_2)} )&:=&
\frac1{c_1c_2}\bigl(c_1+c_2-(s_1-s_2)(\bar s_1 -\bar s_2)\bigr)^2\\
\operatorname{cap}(\mathcal{B}^1_{(s_1,c_1)},\mathcal{B}^2_{(s_2,c_2)} )= \operatorname{cap}(\mathcal{B}^2_{(s_2,c_2)},\mathcal{B}^1_{(s_1,c_1)} )&:=&
\frac1{c_1s_2\bar s_2}(s_1\bar s_2 + \bar s_1 s_2 - c_2)^2\\
\operatorname{cap}(\mathcal{B}^2_{(s_1,c_1)},\mathcal{B}^2_{(s_2,c_2)} )&:=&
\frac1{s_1\bar s_1 s_2 \bar s_2}(s_1\bar s_2 + \bar s_1 s_2)^2
\end{eqnarray*}
Then it turns out that this quantity is indeed invariant under
M\"obius transformations:
\begin{theorem}\label{invariant}
Let $\mathcal B, \tilde{\mathcal B}$ be two circles and $\tau$ a M\"obius transformation.
Then
$$
\operatorname{cap}(\mathcal B,\tilde{\mathcal B}) = 
\operatorname{cap}(\tau(\mathcal B),\tau(\tilde{\mathcal B})) 
$$
\end{theorem}
\begin{proof}
It is easy to check, that $\operatorname{cap}$ is invariant under
translations $z\mapsto \zeta=z+a$, $a\in GF(p^{2m})$, and similarity transformations
$z\mapsto \zeta=b z$, $b\in GF(p^{2m})\setminus\{0\}$. 

The   only   tedious  part   of   the   proof   is  to   check,   that
$\operatorname{cap}$      is       invariant      under      inversion
$\tau: z\mapsto  \zeta=1/z$, since  in this  case, circles  may change
from first  type to second  type and vice  versa: It is  elementary to
check that
$$
\tau(\mathcal B^1_{(s,c)}) = \begin{cases}
\mathcal{B}^1_{\left(\frac{\bar s}{s\bar s-c},\frac c{(s\bar s-c)^2}\right) } &\text{if $s\bar s\neq c$}\\
\mathcal{B}^2_{(\bar s,1)} &\text{if $s\bar s = c$}
\end{cases}
$$
and
$$
\tau(\mathcal B^2_{(s,c)}) = \begin{cases}
\mathcal{B}^1_{\left(\frac{\bar s}c,\frac{s\bar s}{c^2}\right)} &\text{if $c\neq 0$}\\
\mathcal{B}^2_{(\bar s,0)} &\text{if $c = 0$}
\end{cases}
$$
We only carry out the invariance proof for one prototypical case of two circles $\mathcal B^1_{(s_1,c_1)}$ with $s_1\bar s_1 = c_1$
and $\mathcal B^1_{(s_2,c_2)}$ with $s_2\bar s_2 \neq c_2$. Then, in this case,
\begin{eqnarray}\label{cap1}
\lefteqn{
\operatorname{cap}(\mathcal{B}^1_{(s_1,c_1)},\mathcal{B}^1_{(s_2,c_2)} )=}\notag\\
&=&\frac1{c_1c_2}\bigl(c_1+c_2-(s_1-s_2)(\bar s_1 -\bar s_2)\bigr)^2 \notag\\
&=&\frac1{c_1c_2}\bigl(c_2+s_1\bar s_2 +s_2(s_1-s_2)\bigr)^2
\end{eqnarray}
where we have used $s_1\bar s_1=c_1$. On the other hand,
\begin{eqnarray}\label{cap2}
\lefteqn{
\operatorname{cap}\bigl(\tau(\mathcal{B}^1_{(s_1,c_1)}),\tau(\mathcal{B}^1_{(s_2,c_2)}) \bigr)=}\notag\\
&=&\operatorname{cap}\bigl(\mathcal{B}^2_{(\bar s_1,1)},\mathcal{B}^1_{(\frac{\bar s_2}{s_2\bar s_2-c_2},\frac {c_2}{(s_2\bar s_2-c_2)^2})}\bigr)\notag\\
&=&
\frac1{\frac {c_2}{(s_2\bar s_2-c_2)^2 } s_1\bar s_1}\left(\frac{\bar s_2}{s_2\bar s_2-c_2}s_1 + \frac{ s_2}{s_2\bar s_2-c_2}\bar s_1 - 1 \right)^2\notag\\
&=&\frac{(s_2\bar s_2-c_2)^2 }{c_1c_2 }\left(\frac{\bar s_2}{s_2\bar s_2-c_2}s_1 + \frac{ s_2}{s_2\bar s_2-c_2}\bar s_1 - 1 \right)^2
\end{eqnarray}
again since $s_1\bar s_1=c_1$. Obviously, the expressions in~\eqref{cap1} and~\eqref{cap2} agree.
The other cases are similar.
\end{proof}
In a next step  we show that, as it is the case in the
classical M\"obius plane, it is possible to transform any two
non-intersecting M\"obius circles into concentric circles.

\subsection{Transformation of non-intersecting circles into concentric circles}\label{transform}
\begin{theorem}\label{MT}
  Any two disjoint circles in $\mathbb{M}(p^m)$ can be mapped to concentric circles using a suitable M\"obius transformation.
\end{theorem}

We give a combinatorial proof which uses the following results.

\begin{lemma} \label{HowManyTangent} 
  For any given circle $\cB$ there are $\frac{1}{2}(p^{3m}-3p^{2m}+2p^m)$ circles disjoint to $\cB$.
\end{lemma}

\begin{proof} 
  By axiom (M2), for a point $P$ on $\cB$ and any other point $Q$ not on $\cB$, 
  there is a unique circle tangential to $\cB$ through $P$ and $Q$. 
  There are $p^{2m}+1$ points in total and $p^m+1$ points on $\cB$. 
  So, for any of the $p^{2m}-p^m$ points not on $\cB$, there is such a unique circle through a given point $P$ on $\cB$. 
  Since there are $p^m+1$ points on every circle, 
  exactly $p^m$ of these tangent circles are the same. 
  This can be done for every point on $\cB$, 
  which leads to $\frac{(p^m+1)(p^{2m}-p^m)}{p^m}=p^{2m}-1$ circles which are tangential to $\cB$.

  For the circles intersecting $\cB$, note that by axiom (M1), 
  there is a unique circle through two points on $\cB$ and any other point not on $\cB$. 
  So for two fixed points on $\cB$, 
  there are $\frac{p^{2m}-p^m}{p^m-1}=p^m$ circles intersecting $\cB$ in those two points. 
  This can be done for any pair of points on $\cB$, 
  which leads to $\frac{p^{2m}(p^m+1)}{2}$ circles which intersect $\cB$. 
  
  Since there are $p^m(p^{2m}+1)$ circles in total, 
  the number of circles disjoint to $\cB$ is given by
  $$ p^m(p^{2m}+1) - (p^{2m}-1)-\frac{p^{2m}(p^m+1)}{2} -1 = \frac{1}{2}(p^{3m}-3p^{2m}+2p^m)$$
\end{proof}

\begin{lemma}
   There are exactly $p^{3m}-p^m$ M\"obius transformations which map the unit circle $z\overline{z}=1$ to itself.
   In particular, they are given by 
   $$\Phi_1(z) = \frac{bz-ba}{-\overline{a}z+1}$$
   for $b\overline{b}=1$ and $a\overline{a}\neq 1$ and by
   $$ \Phi_2(z)=\frac{b}{z} $$
   for $b\overline{b}=1$.
\end{lemma}

\begin{proof}
  Recall that M\"obius transformations act sharply three times transitive. 
  So, for $\{P_1,P_2,P_3\}$ on the unit circle $z\overline{z}=1$, 
  there are $\binom{p^m+1}{3}$ choices for mapping it to any three points on $z\overline{z}=1$ again.
  Moreover, for any choice of three points, there are $3!=6$ such transformations.
  Hence, there are
    $$ 3!\binom{p^m+1}{3} = p^{3m}-p^m $$
  such M\"obius transformations.
  
  Let us have a closer look at $\Phi_1$. Of course, we have 
    $$ \left(\frac{bz-ba}{-\overline{a}z+1}\right) \overline{\left(\frac{bz-ba}{-\overline{a}z+1}\right)} = 1$$
  whenever $b\overline{b}=1$ and $z\overline{z}=1$. The condition $a\overline{a}\neq 1$ ensures that we do not divide by 0.
  There are $p^m+1$ choices for $b$ and $p^{2m}-(p^m+1)$ choices for $a$. Hence, there are
  $p^{3m}-2p^m-1$ transformations $\Phi_1$.
  
  Similar, $\Phi_2$ maps the unit circle to itself 
  whenever $b\overline{b}=1$, so there are $p^m+1$ such transformations.
  
  Since $(p^{3m}-2p^m-1) + (p^m+1) = p^{3m}-p^m$, these are indeed all such transformations.
\end{proof}

Now we are ready to prove Theorem~\ref{MT}:

\begin{proof}[Proof of Theorem~\ref{MT}]
  There are $p^m-2$ circles concentric to $z\overline{z}=1$, 
  namely all those circles $z\overline{z}=c$ for $c=2,\ldots,p^m-1$.
  We now apply all the M\"obius transformations, which map the unit circle to itself, to those concentric circles.
  
  Note that every image occurs exactly $2(p^m+1)$ times.
  
  To see this, let us first think about the circle $z\overline{z}=c$ for $c\in \{2,\ldots,p^m-1\}$ fixed.
  Clearly, $\Phi_1$ for choosing $a=0$ maps $z\overline{z}=c$ to $z\overline{z}=c$, for all $b\overline{b}=1$.
  Moreover, applying $\Phi_2$ to $z\overline{z}=\frac{1}{c}$ gives $z\overline{z}=c$ as well, for all $b\overline{b}=1$.
  Since for any other choice of $a$ in $\Phi_1$, the center of $z\overline{z}=c$ is translated, 
  the circle $z\overline{z}=c$ occurs $2(p^m+1)$ times 
  when applying all $p^{3m}-p^m$ M\"obius transformations $\Phi_1$ and $\Phi_2$ described above to $z\overline{z}=c$.
  Similarly can be proceeded for other circles $(z-s)(\overline{z}-\overline{s})=c$. 
 
  So, we apply the $p^{3m}-p^m$ M\"obius transformations to the $p^m-2$ circles concentric to $z\overline{z}=1$.
  Every image occurs $2(p^m+1)$ times, i.e.\ we get
  $$ \frac{(p^m-2)(p^{3m}-p^m)}{2(p^m+1)} = \frac{1}{2}(p^{3m}-3p^{2m}+2p^m) $$
  which is, by Lemma \ref{HowManyTangent}, exactly the number of circles disjoint to $z\overline{z}=1$.
\end{proof}

\subsection{General criterion for Steiner chains}
Let $\mathcal B$ and $\tilde{\mathcal B}$ be non-intersecting circles in $\mathbb M(p^m)$.
As we have seen in Section~\ref{transform}, it is possible to transform them
into $\tau(\mathcal B)=\mathcal B^1_{(0,1)}$, $\tau(\tilde{\mathcal B})=\mathcal B^1_{(0,b)}$
by using a suitable M\"obius transformation. By Theorem~\ref{invariant}, we have
$$
c:=\operatorname{cap}(\mathcal B,\tilde{\mathcal B}) = 
\operatorname{cap}(\tau(\mathcal B),\tau(\tilde{\mathcal B})) = 
\operatorname{cap}(\mathcal B^1_{(0,1)}, \mathcal B^1_{(0,b)})=
\frac1{b}(1+b)^2
$$
Solving for $b$ gives $b=\frac12(c-2\pm\sqrt{c(c-4)})$.
Changing between the two possible signs corresponds to applying
an additional inversion $z\mapsto 1/z$ to both circles, and we may take the 
plus sign by convention. Then, the general criterion follows from Theorem~\ref{concentric}:
\begin{theorem}
Let  $\mathcal B$ and $\tilde{\mathcal B}$ be non-intersecting circles,
$c:=\operatorname{cap}(\mathcal B,\tilde{\mathcal B})$ and
 $b:=\frac12(c-2+\sqrt{c(c-4)})$. Then,
  if $b = \mu^2$, for $\mu$ in $GF(p^m)$,
  $\mathcal B$ and $\tilde{\mathcal B}$  carry a proper Steiner chain of length $k$ if and only if the following conditions are satisfied.
    \begin{enumerate}
      \item $-\mu$ is a non-square in $GF(p)$,
      \item $\mu$ solves $ P_1^k=1$ for $P_1$ given by 
        \begin{align}
          P_1 = \frac{-\mu^2+6\mu-1 + 4(\mu-1)\sqrt{-\mu}}{(1+\mu)^2}
        \end{align}
            but $P_1^l \neq 1$ for all $1 \leq l \leq k-1$.
    \end{enumerate}
\end{theorem}


\section{Comparison to the Euclidean plane}

Steiner's Porism is well understood in the Euclidean plane.
For two concentric circles of radius $1$ and $R$,
a proper Steiner chain of length $k$ 
which wraps around the inner circle once
can be constructed if and only if 
$$ R = \frac{1 + \sin(\phi)}{1-\sin(\phi)} $$
where $\phi = \frac{\pi}{k}$.

For some values of $k$, 
we can express $\sin(\frac{\pi}{k})$ explicitly in terms of radicals.
In the following table, we calculate $R$ for some values of $k$.

\begin{table}[h]
  \begin{center}
    \begin{tabular}{|c|c|c|c|}
       \hline
       $k$ & \small{$\frac{180}{k}$} & $\sin{\frac{\pi}{k}}$ & $R$ \\[0.1 cm]
       \hline
       \hline
       $3$ & $60$ & $\frac{\sqrt{3}}{2}$ & $7+4\sqrt{3}$ \\[0.1 cm]
      \hline
        $4$ & $45$ & $\frac{\sqrt{2}}{2}$ & $3+2\sqrt{2}$ \\[0.1 cm]
      \hline
        $5$ & $36$ & $\frac{\sqrt{10-2\sqrt{5}}}{4}$ & $11 - 4 \sqrt{5} + 2 \sqrt{50 - 22 \sqrt{5}}$ \\[0.1 cm]
      \hline
        $6$ & $30$ & $\frac{1}{2}$ & $3$ \\
       \hline
          $8$ & $22.5$ & $\frac{\sqrt{2-2\sqrt{2}}}{2}$ & $7 - 4 \sqrt{2} + 2 \sqrt{20 - 14 \sqrt{2}}$ \\[0.1 cm]
      \hline
    \end{tabular}
  \end{center}
  \caption{Some values for Steiner chains of length $k$.}
  \label{tab: some values for Sin}
\end{table}

Note that our values for $R$ coincide with the values for $\mu$ calculated in
the Lemmas \ref{lemma: 3 chain}, \ref{lemma: 4 chain}, 
\ref{lemma: 5 chain}, \ref{lemma: 6 chain} and \ref{lemma: 8 chain}.

Let us have a closer look at why this is the case. 
Recall that for two common tangent circles of $\cB_1$ and $\cB_b$,
we calculated their point of intersection in Lemma \ref{lemma: P1 and P2}, namely
  \begin{align} 
      P = \frac{(-\mu^2+6\mu-1) + 4(\mu-1)\sqrt{-\mu}}{(1+\mu)^2}.
  \end{align}
For the construction of a Steiner chain, we needed $-\mu$ to be a non-square in $GF(p^m)$.
So, let us rewrite $P$ as
  \begin{align} 
      P = \frac{-\mu^2+6\mu-1}{(1+\mu)^2} + \sqrt{-\mu} \ \frac{4(\mu-1)}{(1+\mu)^2}.
  \end{align}
  
Note that $\frac{-\mu^2+6\mu-1}{(1+\mu)^2}$ and $\frac{4(\mu-1)}{(1+\mu)^2}$ are both in $GF(p^m)$.
By assumption, $\sqrt{-\mu}$ is not in $GF(p^m)$, 
so let us consider $P$ to be an element of
$GF(p^m)(\sqrt{-\mu})$, which is isomorphic to $GF(p^{2m})$.
So, we can write all elements of $GF(p^m)(\sqrt{-\mu})$ in the form $z = a+b\sqrt{-\mu}$.
Let us refer to $a$ as the \emph{real part} of $z$, denoted by $\Re(z)$, 
and to $b$ as the \emph{imaginary part} of $z$, denoted by $\Im(z)$.

Now, recall that for a Steiner chain of length $k$
the tangent points on $\cB_1$ are given by
$1$, $P$, $P^2$, $\ldots$, $P^{k-1}$. 
The real part of $P^2$ satisfies
$$ \Re(P^2) = 2 \Re(P)^2 - 1 $$
and the imaginary part of $P$ satisfies
$$ \Im(P^2) = 2 \Re(P) \Im(P).$$
Note that those equations are the same as for the sine and cosine in the Euclidean plane,
namely
$$ \cos(2 \phi) = 2 \cos(\phi)^2-1 $$
and
$$ \sin(2 \phi) = 2 \cos(\phi) \sin(\phi).$$
Hence, calculating $P_1^2$ is in a sense the same as doubling the angle between $1$ and a point on $\cB_1$.


\end{document}